\documentclass[a4paper,twoside,12pt]{article}

\usepackage{amssymb,amsmath,amsthm,latexsym}
\usepackage{amsfonts}
\usepackage{graphicx,caption,subcaption}
\usepackage[pdftex,bookmarks,colorlinks=false]{hyperref}
\usepackage[hmargin=1.2in,vmargin=1.2in]{geometry}
\usepackage[mathscr]{euscript}

\usepackage[auth-lg,affil-sl]{authblk}
\usepackage{float}
\usepackage{calc,tikz}
\usetikzlibrary{arrows,decorations.markings}
\tikzstyle{vertex}=[circle, draw, inner sep=1pt, minimum size=8pt]
\newcommand{\vertex}{\node[vertex]}

\newtheorem{theorem}{Theorem}[section]
\newtheorem{definition}[theorem]{Definition}

\newcommand{\noi}{\noindent}

\newcommand{\f}{\mathsf{F}}
\newcommand{\N}{\mathbb{N}}

\newcommand{\cC}{\mathcal{C}}
\newcommand{\J}{\mathscr{J}}

\title{\sc On $J$-Colorability of Certain Derived Graph Classes}

\author{{\sc Federico Fornasiero}}
\affil{\small Department of mathemathic\\ Universidade Federal de Pernambuco \\ Recife, Pernambuco, Brasil\\{\tt federico@dmat.ufpe.br}}

\author{{\sc Sudev Naduvath}}
\affil{\small Centre for Studies in Discrete Mathematics\\ Vidya Academy of Science \& Technology \\ Thrissur, Kerala, India.\\{\tt sudevnk@gmail.com}}

\date{}

\begin{document}
\maketitle

\begin{abstract}
A vertex $v$ of a given graph $G$ is said to be in a rainbow neighbourhood of $G$, with respect to a proper coloring $\cC$ of $G$, if the closed neighbourhood $N[v]$ of the vertex $v$ consists of at least one vertex from every colour class of $G$ with respect to $\cC$. A maximal proper colouring of a graph $G$ is a $J$-colouring of $G$ if and only if every vertex of $G$ belongs to a rainbow neighbourhood of $G$. In this paper, we study certain parameters related to $J$-colouring of certain Mycielski type graphs.
\end{abstract}

\noi \textbf{Key Words}: Mycielski graphs, graph coloring, rainbow neighbourhoods, $J$-coloring of graphs. 

\vspace{0.3cm}

\noi \textbf{Mathematics Subject Classification 2010}: 05C15, 05C38, 05C75.


\section{Introduction}

For general notations and concepts in graphs and digraphs we refer to \cite{BM1,FH1,DBW}. For further definitions in the theory of graph colouring, see \cite{CZ1,JT1}. Unless specified otherwise, all graphs mentioned in this paper are simple, connected  and undirected graphs. 

\subsection{Mycielskian of Graphs}

Let $G$ be a triangle-free graph with the vertex set $V(G)=\{v_1,\ldots,v_n\}$. The \textit{Mycielski graph} or the \textit{Mycielskian} of a graph $G$, denoted by $\mu(G)$, is the graph with vertex set $V(\mu(G))=\{v_1,v_2,\ldots, v_n, u_1,u_2, \ldots, u_n, w\}$ such that $v_iv_j\in E(\mu(G))\iff v_iv_j\in E(G)$, $v_iu_j\in E(\mu(G))\iff v_iv_j\in E(G)$ and $u_iw\in E(\mu(G))$ for all $i=1,\ldots, n$.

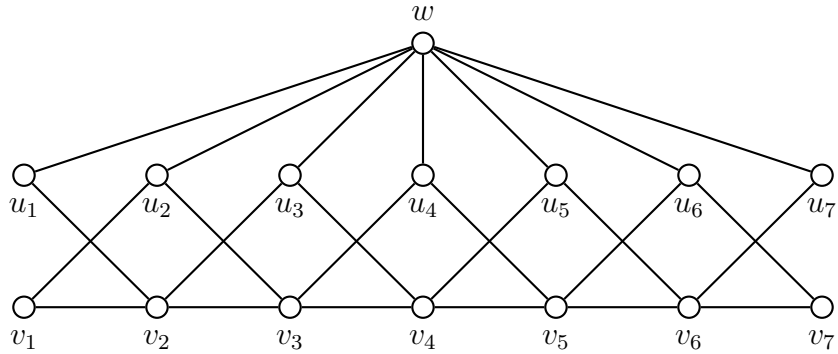
\begin{figure}[h!]
\centering
\begin{tikzpicture}[auto,node distance=1.75cm,
thick,main node/.style={circle,draw,font=\sffamily\Large\bfseries}]

\vertex [label=below:$v_1$] (1) {};
\vertex [label=below:$v_2$] (2) [right of=1] {};
\vertex [label=below:$v_3$] (3) [right of=2] {};
\vertex [label=below:$v_4$] (4) [right of=3] {};
\vertex [label=below:$v_5$] (5) [right of=4] {};
\vertex [label=below:$v_6$] (6) [right of=5] {};
\vertex [label=below:$v_7$] (7) [right of=6] {};
\vertex [label=below:$u_1$] (8) [above of=1] {};
\vertex [label=below:$u_2$] (9) [above of=2] {};
\vertex [label=below:$u_3$] (10) [above of=3] {};
\vertex [label=below:$u_4$] (11) [above of=4] {};
\vertex [label=below:$u_5$] (12) [above of=5] {};
\vertex [label=below:$u_6$] (13) [above of=6] {};
\vertex [label=below:$u_7$] (14) [above of=7] {};
\vertex [label=above:$w$] (15) [above of=11] {};
\path[every node/.style={font=\sffamily\small}]
(1) edge node {} (2)
edge node {} (9)
(2) edge node {} (3)
edge node {} (8)
edge node {} (10)
(3) edge node {} (4)
edge node {} (9)
edge node {} (11)
(4) edge node {} (5)
edge node {} (10)
edge node {} (12)
(5) edge node {} (6)
edge node {} (11)
edge node {} (13)
(6) edge node {} (7)
edge node {} (12)
edge node {} (14)
(7) edge node {} (13)
(15) edge node {} (8)
edge node {} (9)
edge node {} (10)
edge node {} (11)
edge node {} (12)
edge node {} (13)
edge node {} (14);
\end{tikzpicture}
\caption{The Mycielski graph $\mu(P_7)$}
\end{figure}

In the above mentioned conditions of Mycielski graphs, we call the two vertices $v_i,$ $u_i$ \textit{twin vertices} and the vertex $w$ is called the \textit{root vertex} of the Mycielskian $\mu(G)$. 

By a \textit{Mycielski type graph}, we mean a graph that can be constructed from the Mycielski graphs or the graphs generated from a given graphs using some or similar rules of constructing their Mycielski graphs. 

\subsection{$J$-Coloring of Graphs}

The notion of $J$-coloring of a graph, has been defined for the first time in \cite{NKS} as follows:

\begin{definition}\label{Def-1.1}{\rm
\cite{NKS} A graph $G$ is said to have a \emph{$J$-colouring} $\mathcal{C}$ if it has the maximal number of colours such that if every vertex $v$ of $G$ belongs to a rainbow neighbourhood of $G$. The number of colours in a $J$-coloring $\cC$ of $G$ is called the \emph{$J$-colouring number} of $G$.
}\end{definition}

\begin{definition}\label{Def-1.2}{\rm
\cite{NKS} A graph $G$ is said to have a \emph{$J^*$-colouring} $\mathcal{C}$ if it has the maximal number of colours such that if every internal vertex $v$ of $G$ belongs to a rainbow neighbourhood of $G$. The number of colours in a $J^*$-coloring $\cC$ of $G$ is called the \emph{$J^*$-colouring number} of $G$.
}\end{definition}

It can be noted that all graphs, in general, need not have a $J$-coloring. Hence, the studies on the graphs which admit $J$-coloring and their properties and structural characterisations attract much interests. Some studies in this direction can be seen in \cite{KS1,NKS}. 

The initial purpose of this paper is to study the $J$-colourability of the Mycielskian and certain Mycielski type graphs of some fundamental graph classes.


\section{$J$-Colourability of Mycielski Graphs}

Note that the Mycielski graph $\mu(G)$ of a graph $G$ has no pendant vertices and hence the $J$-colouring and the $J^*$-colouring of the Mycielski graphs $\mu(G)$, if they exist, are the same. We first try to repeat the original demonstration of Mycielskian graph. To do that, we have to fix a $J$-colouring on the graph $G$.

\begin{theorem}
Let $G$ be a graph with $J$-colouring $\cC=\{c_1,c_2,c_3\ldots, c_k\}$. Then, the graph $\mu(G)$ is not $J$-colourable (and so, it is not $J^*$-colourable).
\end{theorem}
\begin{proof}
Note that $\mu(K_2)$ is isomorphic to $C_5$ and hence it is not $J-colourable$ (as it is proved in \cite{NKS}). Hence, we can consider graphs with order greater than $2$. Let us assume that we can have a new $J$-colouring $\cC=\{c_1,c_2\ldots, c_j\}$ on $\mu(G)$. Without loss of generality, we can assume now that the colour of $w$ is $c'_1$. Every vertex $u_i$ have to be coloured with one of the other colour $c_2,\ldots, c_j $. Hence, there exists at least a vertex $v_1$ with colour $c_1$. Let be $v_2$ the vertex connected to $v_1$ such as the twin vertex $u_2$ has the colour $c_2$. Here, we have the following two possibilities:

\begin{enumerate}\itemsep0mm
\item[(i)]: If the colour of $v_2$ is different from the previous ones, let us say $c_3$, we have that for the definition of rainbow neighbourhood even the twin vertex $u_2$ has to be connected with a vertex with the same colour, and it has to be a vertex $v_3$ because no one of the vertices $u_j$ are connected, and $w$ has the colour $c_1$. But if it is so, than for the construction of $\mu(G)$ even the vertex $v_2$ has to be connected with $v_3$, and so it is not a proper coloring because two vertices has the same colours (see Figure~\ref{case1}).

\begin{figure}[H]
\centering
\begin{tikzpicture}[auto,node distance=1.75cm,
thick,main node/.style={circle,draw,font=\sffamily\Large\bfseries}]
\vertex [label=left:$v_1$] (1) {\small{$c_1$}};
\vertex [label=right:$v_2$] (2) [right of=1] {\small{$c_3$}};
\vertex [label=left:$v_3$] (3) [left of=1] {\small{$c_3$}};
\vertex[label=left:$u_1$] (4) [above of=1] {\small{$c_2$}};
\vertex [label=right:$u_2$] (5) [right of=4] {\small{$c_2$}};
\vertex [label=above:$w$] (6) [above of=4] {\small{$c_1$}};
\path[every node/.style={font=\sffamily\small}]
(1) edge node {} (5)
edge node {} (2)
(2) edge node {} (4)
edge [bend left, dashed]  node {} (3)
(3) edge node {} (5)
(4) edge node {} (6)
(5) edge node {} (6);
\end{tikzpicture}
\caption{Case (i)}
\label{case1}
\end{figure}

\item[(ii)] If the colour of $v_2$ is $c_2$, the twin vertices $u_1$ of $v_1$ has to have a different colour (let us say $c_3$), because it is linked to $w$ that has the colour $c'_1$ and to $v_2$ that has the colour $c_2$. But, in this case the vertex $v_1$ has to be connected with another vertex which has the colour $c_3$.

So we have to differentiate two different possibilities:

\item[(ii)(a)] If the vertex $v_1$ is connected to a vertex $v_3$ who has the colour $c_3$ for the construction of $\mu(G)$ even the twin vertices $u_1$ is connected to $v_3$ and so it is not a proper coloring because two connected vertices have the same colour (see Figure~\ref{case2a}).

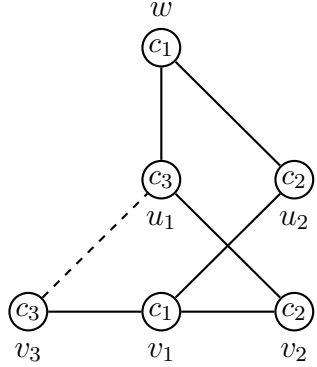
\begin{figure}[H]
\centering
\begin{tikzpicture}[auto,node distance=1.75cm,
thick,main node/.style={circle,draw,font=\sffamily\Large\bfseries}]

\vertex [label=below:$v_1$] (1) {\small{$c_1$}};
\vertex [label=below:$v_2$] (2) [right of=1] {\small{$c_2$}};
\vertex [label=below:$v_3$] (3) [left of=1] {\small{$c_3$}};
\vertex[label=below:$u_1$] (4) [above of=1] {\small{$c_3$}};
\vertex [label=below:$u_2$] (5) [right of=4] {\small{$c_2$}};
\vertex [label=above:$w$] (6) [above of=4] {\small{$c_1$}};
\path[every node/.style={font=\sffamily\small}]
(1) edge node {} (5)
edge node {} (2)
edge node {} (3)
(2) edge node {} (4)
(3) edge [dashed] node {} (4)
(4) edge node {} (6)
(5) edge node {} (6);

\end{tikzpicture}
\caption{Case (ii)(a)}
\label{case2a}
\end{figure}

\item[(ii)(b)] If the vertex $v_1$ is connected to a vertex $u_3$ which has the colour $c_3$, first we can note that the twin vertex $v_3$ cannot have a new colour, because it would lead to a contradiction over $u_3$ similar to the case (i).

Note that $v_3$ cannot have the colour $c_1$ (because for construction it is connected with $v_1$) nor the colour $c_3$ (because always for construction it is connected to the twin vertex $u_1$) so it has to be coloured with the colour $c_2$. But if it is so, $u_3$ needs to be connected with a vertex $v_i$ whose colour is $c_2$, but it cannot be the twin vertex $v_3$ for the construction, nor the vertex $v_2$ because it will lead to have the triangle $v_1v_2u_3v_1$. If the graph $G$ has only 3 nodes we reach a contradiction yet, if it is not let us call $v_4$ the vertex with colour $c_2$ connected to $u_3$. For the construction of $\mu(G)$ it has to be connected to the vertex $v_3$ and it finally leads to a contradiction because the two vertices would have the same colour (see Figure~\ref{case2b}).

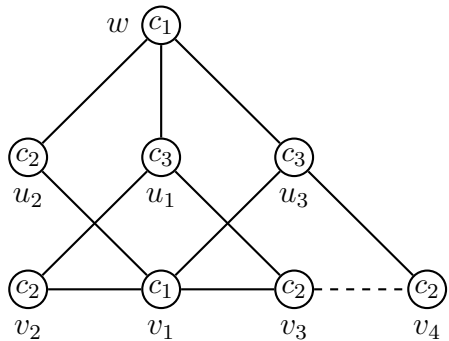
\begin{figure}[h!]
\centering
\begin{tikzpicture}[auto,node distance=1.75cm,
thick,main node/.style={circle,draw,font=\sffamily\Large\bfseries}]

\vertex [label=below:$v_1$] (1) {\small{$c_1$}};
\vertex [label=below:$v_2$] (2) [left of=1] {\small{$c_2$}};
\vertex [label=below:$v_3$] (3) [right of=1] {\small{$c_2$}};
\vertex [label=below:$u_1$] (4) [above of=1] {\small{$c_3$}};
\vertex [label=below:$u_2$] (5) [left of=4] {\small{$c_2$}};
\vertex [label=left:$w$] (6) [above of=4] {\small{$c_1$}};
\vertex [label=below:$u_3$] (7) [above of=3] {\small{$c_3$}};
\vertex [label=below:$v_4$] (8) [right of=3] {\small{$c_2$}};
\path[every node/.style={font=\sffamily\small}]
(1) edge node {} (5)
edge node {} (2)
edge node {} (3)
edge node {} (7)
(2) edge node {} (4)
(3) edge node {} (4)
edge [dashed] node {} (8)
(4) edge node {} (6)
(5) edge node {} (6)
(7) edge node {} (6)
edge node {} (8);
\end{tikzpicture}
\caption{Case (ii)(b)}
\label{case2b}
\end{figure}
\end{enumerate} %
Hence, the Mycielskian graph of any graph $G$ is not $J$-colourable, irrespective of whether the $G$ is $J$-colourable or not.
\end{proof}

\section{Some New Constructions}

Since Mycielskian of any graph does not admit a $J$-colouring, our immediate aim is to construct some simple connected graphs from certain given graphs such that new graphs also admit an extended $J$-coloring. In this section, we discuss the $J$-colorability of certain newly constructed Mycielski type graphs of a given graphs.

The first one among such graphs is the \textit{crib graph}, denoted by $c(G)$, of a graph $G$, which is defined in \cite{SSS} as follows:

\begin{definition}{\rm
\cite{SSS}The \textit{crib graph}, denoted by $c(G)$, of a graph $G$ is the graph whose vertex set is $V(\mu(G))=\{v_1,v_2,\ldots, v_n, u_1,u_2, \ldots, u_n, w\}$ such that $v_iv_j\in E(\mu(G))\iff v_iv_j\in E(G)$, $v_iu_j\in E(\mu(G))\iff v_iv_j\in E(G)$ and $v_iw,u_iw\in E(\mu(G))$ for all $i=1,\ldots, n$.
}\end{definition}

\noi Figure \ref{c(G)} depicts the crib graph of $P_6$.

\begin{figure}[h!]
\centering
\begin{tikzpicture}
\vertex (1) at (0,0) [label=below:$v_{1}$] {};
\vertex (2) at (2,0) [label=below:$v_{2}$] {};
\vertex (3) at (4,0) [label=below:$v_{3}$] {};
\vertex (4) at (6,0) [label=below:$v_{4}$] {};
\vertex (5) at (8,0) [label=below:$v_{5}$] {};
\vertex (6) at (10,0) [label=below:$v_{6}$] {};
\vertex (7) at (0,2) [label=left:$u_{1}$] {};
\vertex (8) at (2,2) [label=left:$u_{2}$] {};
\vertex (9) at (4,2) [label=below:$u_{3}$] {};
\vertex (10) at (6,2) [label=below:$u_{4}$] {};
\vertex (11) at (8,2) [label=right:$u_{5}$] {};
\vertex (12) at (10,2) [label=right:$u_{6}$] {};
\vertex (13) at (5,4) [label=above:$w$] {};
\path
(1) edge (2)
(2) edge (3)
(3) edge (4)
(4) edge (5)
(5) edge (6)
(7) edge (2)
(8) edge (1)
(8) edge (3)
(9) edge (2)
(9) edge (4)
(10) edge (3)
(10) edge (5)
(11) edge (4)
(11) edge (6)
(12) edge (5)
(1) edge [bend right=4] (13)
(2) edge [bend left=5] (13)
(3) edge (13)
(4) edge (13)
(5) edge [bend right=5] (13)
(6) edge [bend left=4] (13)
(7) edge (13)
(8) edge (13)
(9) edge (13)
(10) edge (13)
(11) edge (13)
(12) edge (13)
;
\end{tikzpicture}
\caption{Crib graph of $P_6$}
\label{c(G)}
\end{figure}
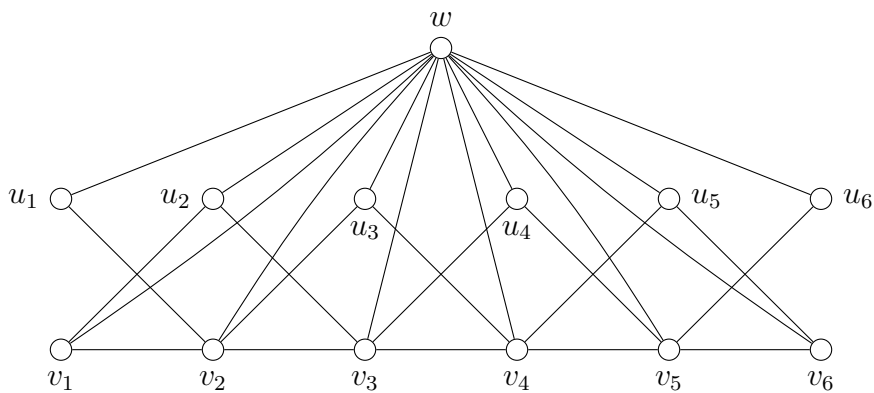

The following theorem discusses the admissibility of an extended $J$-coloring by the crib graph of a $J$-colorable graph $G$.

\begin{theorem}\label{Thm-3.2}
The crib graph $c(G)$ of a $J$-colourable graph $G$ is also $J$-colourable. Also, $\J(c(G))=\J(G)+1$.
\end{theorem}
\begin{proof}
Assume that the graph $G$ under consideration admits a $J$-coloring, say $\cC=\{c_1,c_2,\ldots,c_k\}$, where $k=\chi(G)$, the chromatic number of $G$. While coloring the vertices of $c(G)$, we notice the following points:
\begin{enumerate}\itemsep0mm
\item[(i)] Since, the twin vertices $u_i$ and $v_i$ in $c(G)$ are adjacent to each other, both of them can have the same color. 
\item[(ii)] Since $N(u_i)=N(v_i)$ for all $1\le i\le n$, it follows that $N[u_i]$ is also a rainbow neighbourhood in $c(G)$. Therefore, the subgraph of $c(G)$ induced by the vertex set $\{v_1,v_2,\ldots, v_n,u_1,u_2,\ldots,u_n\}$ admit the same $J$-coloring $\cC$.
\item [(iii)] Since the root vertex $w$ is adjacent to other vertices in $c(G)$, it cannot have any color from $\cC$. Therefore, we need a new color, say $c_{k+1}$ to color the vertex $w$.
\item[(iv)] Since the root vertex $w$ is adjacent to other vertices in $c(G)$, it belongs to a rainbow neighbourhood in $c(G)$ and will not influence the belongingness of other vertices to some rainbow neighbourhoods in $c(G)$.
\end{enumerate}

In view of the conditions mentioned above, notice that $\cC\cup\{c_{k+1}\}$ is a $J$-coloring of $c(G)$ and $\J(c(G))=k+1=\J(G)+1$. This completes the proof.
\end{proof}

Another similar graph that catches attention in this context is the shadow graph of a graph $G$. The \textit{shadow graph} of a graph $G$, denoted by $s(G)$, is the graph $G$ is the graph obtained from its Mycielski graph $\mu(G)$ by removing the root vertex.

\noi The following theorem discusses the admissibility of a $J$-coloring by the shadow graph $s(G)$ of a $J$-colorable graph $G$. 

\begin{theorem}
The shadow graph $s(G)$ of a $J$-colorable graph $G$ is also $J$-colorable. Moreover, $\J(s(G))=\J(G)$.
\end{theorem}
\begin{proof}
The proof is immediate from the proof of Theorem \ref{Thm-3.2}.
\end{proof} 

Next, we construct a new graph $\f(G)$ from a triangle-free, simple and connected graph $G$ such that $\f(G)$ has $J$-chromatic number $k+1$ when $G$ has $J$-chromatic number $k$. The construction is described below.

\begin{definition}{\rm 
Let $G$ be a triangle-fee graph, with $V(G)=\{v_1,\ldots, v_n\}$. We define the \emph{Federico graph} $\f(G)$ of $G$ as the graph such that $V(\f(G))=\{v_1,v_2\ldots, v_n, u_1,u_2\ldots, \\ u_n, w_1,w_2\ldots, w_n\}$ and with edges that follows the rules:
\begin{enumerate}\itemsep0mm
\item[(i)] $v_iv_j\in E(\f(G))\iff v_iv_j\in E(G)$
\item[(ii)] $w_iw_j\in E(\f(G))\iff v_iv_j\in E(G)$
\item[(iii)] $u_iw_j\in E(\f(G))\iff v_iv_j\in E(G)$
\item[(iv)] for all $i=1,\ldots, n$, $v_iu_i\in V(\f(G))$
\end{enumerate}
}\end{definition}

The following figure illustrates the Federico graph of the graph $P_5$. 

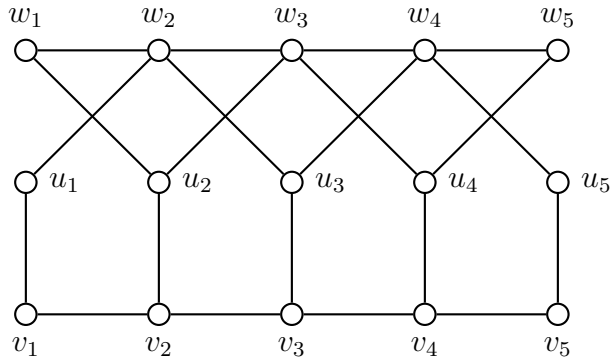
\begin{figure}[h!]
\centering
\begin{tikzpicture}[auto,node distance=1.75cm,
thick,main node/.style={circle,draw,font=\sffamily\Large\bfseries}]

\vertex [label=below:$v_1$] (1) {};
\vertex [label=below:$v_2$] (2) [right of=1] {};
\vertex [label=below:$v_3$] (3) [right of=2] {};
\vertex [label=below:$v_4$] (4) [right of=3] {};
\vertex [label=below:$v_5$] (5) [right of=4] {};
\vertex [label=right:$u_1$] (6) [above of=1] {};
\vertex [label=right:$u_2$] (7) [above of=2] {};
\vertex [label=right:$u_3$] (8) [above of=3] {};
\vertex [label=right:$u_4$] (9) [above of=4] {};
\vertex [label=right:$u_5$] (10) [above of=5] {};
\vertex [label=above:$w_1$] (11) [above of=6] {};
\vertex [label=above:$w_2$] (12) [above of=7] {};
\vertex [label=above:$w_3$] (13) [above of=8] {};
\vertex [label=above:$w_4$] (14) [above of=9] {};
\vertex [label=above:$w_5$] (15) [above of=10] {};
\path[every node/.style={font=\sffamily\small}]
(1) edge node {} (2)
edge node {} (6)
(2) edge node {} (3)
edge node {} (7)
(3) edge node {} (4)
edge node {} (8)
(4) edge node {} (5)
edge node {} (9)
(5) edge node {} (10)
(11) edge node {} (12)
edge node {} (7)
(12) edge node {} (13)
edge node {} (6)
edge node {} (8)
(13) edge node {} (14)
edge node {} (7)
edge node {} (9)
(14) edge node {} (15)
edge node {} (8)
edge node {} (10)
(15) edge node {} (9);
\end{tikzpicture}
\caption{The Federico Graph $\f(P_5)$}
\end{figure}

First we can note that the graph $\f(G)$ has no pendant vertices and so the $J$-colouring od $\f(G)$, if exists, coincides to the $J^*$-colouring. This fact is straight forward.

\begin{theorem}\label{teoM}
Let $G$ be a $J$-colourable, triangle-free graph of order $n$ with $J$-colouring number $k$. Then the graph $\f(G)$ is triangle-free and with higher $J$-colouring number. If $\J(G)=k$, then $\J(\f(G))=k+1$.
\end{theorem}
\begin{proof}
First of all we can see that no pair of vertices $u_i$ is connected, therefore no triangle can involve a pair of these vertices. Also, no vertex $w_i$ is connected to a vertex $v_j$.

Remembering that $G$ is triangle-free, it is not possible that three vertices $v_i$ are connected in $\f(G)$ too. Similarly for the vertices $w_i$ that form between them a copy of the graph $G$. Hence, we have only two possibilities left:
\begin{enumerate}\itemsep0mm
\item[(i)] if $v_i$ is connected to $v_j$ we have that $u_i$ is connected to $v_i$ but not to $v_j$, by construction, so no triangle of this type is involved. 
\item[(ii)] if $w_i$ is connected to $w_j$ we have that $u_i$ is connected to $w_j$ but not to $w_i$, so it is proved that $\f(G)$ is triangle-free. 
\end{enumerate}

To construct a proper $J$-colouring on $\f(G)$, let consider a proper $J$-colouring $\varphi:V(G)\rightarrow\{c_1,\ldots, c_k\}$ and let us construct $\varphi^*:V(\f(G))\rightarrow\{c_1,\ldots, c_k, c_{k+1}\}$ by setting:
\begin{enumerate}\itemsep0mm
\item[(i)] $\varphi^*(v_i)=f^*(w_i)=f(v_i)$ for all $i=1,\ldots, n$
\item[(ii)] $\varphi^*(u_i)=c_{k+1}$ for all $i=1,\ldots, n$
\end{enumerate}

First, we have to prove that it is a proper $J$-colouring of $\f(G)$. We note that every vertex $v_i$ has a rainbow neighbourhood in $G$, and hence it has in $\mu(G)$ with one more colour (the colour of $u_i$). Every $w_i$ has the same rainbow of the twin $v_i$ because it is connected with the same vertices connected to $v_i$, and it is connected at least to one of the vertex $u_j$. Finally, every $u_i$ has a $k+1$ rainbow neighbourhood of because it is connected with every $v_i$ and with every $w_j$ that are the connections of $v_i$ in the original graph, so by the definition of $\varphi^*$, every $u_i$ has the same rainbow neighbourhood of $v_i$. Hence, this colouring define a proper $J$-colouring of $G$, it remains  to prove that this colouring is maximal.

Hence, let us assume that there exists a proper $J$-colouring of $\f(G)$ such that $\J(\f(G))=2$. In this case, we can assume that not every $u_i$ has the same colour because if not every $u_i$ is connected only to $v_i$, and every $v_i$ has a rainbow neighbourhood of order $k+2$ and can't have the same colour of the vertices $u_i$. But it would mean that the graph $G$ was $(k+1)-J-$colouring.

Hence, let us start considering the vertex $u_i$. If we prove that independent from the choice of the colour of $u_i$, it is necessary that every $u_i$ has the same colour, for what we have just proved, it follows that the coloring is maximal.

From the above choice of the coloring assignment $\varphi^*$, we note that $\f(G)$ requires at least one more color in its proper coloring than the corresponding proper coloring of the graph $G$. Now, note that the upper bound for the $J$-chromatic number of a graph $G$ is $\delta(G)+1$ (see \cite{NKS}). Since $\delta(\f(G))=\delta(G)+1$, any $J$-coloring of $\f(G)$ can have at most one more color than the $J$-coloring of $G$. From these two conditions, we can conclude that the coloring $\varphi^*$ defined above is a maximal coloring of $\f(G)$ such that every vertex of $\f(G)$ belongs to some rainbow neighbourhood of $\f(G)$. Then, we have $\J(\f(G))=\J(G)+1$, completing the proof. 
\end{proof}

Hence, we have found an interesting construction to have new triangle free graphs with higher $J$-colouring number. In the following theorem, we study what happens to the chromatic number of a Federico graph.

\begin{theorem}
Let $G$ be a graph and $\f(G)$ its Federico graph. Then, $\chi(G)=\chi(\f(G))$
\end{theorem}
\begin{proof}
Let $f:V(G)=\{v_1,\ldots, v_n\}\rightarrow {c_1,\ldots, c_k}$ be a coloring of the vertices of $G$. Let us consider the coloring $g:V(\f(G))=\{v_1,\ldots, v_n, u_1,\ldots, u_n, w_1,\ldots, w_n\}\rightarrow\{c_1, \ldots, c_k\}$ defined by:

\begin{enumerate}\itemsep0mm
\item[(i)] $g(u_i)=g(w_i)=f(v_i)$ for all $i=1,\ldots, n$.
\item[(ii)] if $f(v_i)=c_h$ then $g(v_i)=c_{h+1}$ for all $i=1,\ldots, n$ and $h=1,\ldots, k-1$
\item[(iii)] if $f(v_i)=c_k$ then $g(v_i)=c_1$ for all $i=1,\ldots, n$.
\end{enumerate}

To prove that it is a proper coloring, first we can note that it is a proper coloring on the vertices $v_i$ because $f$ was a proper coloring of $G$ and we have only permutated the colours, and also it is a proper coloring on the vertices $w_i$ because it is a copy of the graph and we have coloured in the same way. So it only left to see that we cannot have the same colour with connections with a vertex $u_i$.

But, $v_i$ is only connected to $u_i$ and they have different colours because $g(u_i)=f(v_i)$ but $c_{h+1}=g(v_i)\not= f(v_i)=c_h$ for the definition of $g$. Also, because each $w_i$ is connected to every $u_j$ such that $v_j\in V(G)$ was connected to $v_i\in V(G)$ and none of which has the same colour $g(u_i)=f(u_i)$ no conflicts arise here.

Hence, we have constructed a proper coloring of $\f(G)$ with the same number of colours of $G$, as claimed. 
\end{proof}

Now we want to study another important colouring property of the Modified Mycielski graph, the \textit{circular chromatic number}. It was first studied in \cite{AV} with the name of star chromatic number, and later in \cite{XZ} provided a comprehensive survey. 

Let $G$ be a graph. For two positive numbers $k,d$ with $k\geq 2d$, we define a $(k,d)$-colouring as the function $f:V(G)\rightarrow \{0,1,\ldots, k-1\}$ such that if two vertices $u,v$ are adjacent, then $|f(u)-f(v)|_k\leq d$ where $|a-b|_k=\mathrm{min}\{|a-b|, k-|a-b|\}$. Then, the \textit{circular chromatic number} of $G$ is defined as $$\chi_c(G):={\rm inf}\left\{\frac{k}{d}\mid G\,\, \mathrm{has}\,\, \mathrm{a}\,\,(k,d)-\mathrm{coloring} \right\}$$

In \cite{XZ} it is shown that if the graph $G$ has at least one edge, then the infimum can be replaced with the minimum and we have $\chi(G)-1\leq \chi_c(G)\leq \chi(G)$.

The circular chromatic number is hard to compute in Mycielski graphs and there's not yet a general formula that compute $\chi_c(\mu(G))$ knowing the circular chromatic number of $G$. But, in the case of Federico graph, we have 

\begin{theorem}
Let $G$ be a graph with $\chi_c(G)$. Then, $\chi_c(\f(G))=\chi_c(G)$.
\end{theorem}  
\begin{proof}
Let $G$ be a graph with a $(k,d)$ coloring $f$ over $V(G)=\{v_1,v_2 \ldots, v_n\}$. Then, we construct the coloring $f^*$ over $ \f(G)$ as follows:
\begin{enumerate}\itemsep0mm
\item[(i)] $f^*(v_i)=f(v_i)$ for all $i=1,\ldots, n$
\item[(ii)] $f^*(u_i)=f^*(w_i)=f(v_i)-d\, \mathrm{mod}\,k$ for all $i=1,\ldots, n$
\end{enumerate}
To see that it's a proper $(k,d)$ coloring of $\f(G)$ we first note that between it's a proper coloring over the vertices $v_i$'s (because it was on $G$) and over the vertices $w_i$'s (because in the construction we have simply added the distance modulo $k$, and their connections are the same than the connections over the vertices $v_i$). A vertex $v_i$ is adjacent only to the vertex $u_i$ and so by construction it has exactly distance $d$.

The vertex $u_i$ is connected to every vertex $w_j$ such that $v_iv_j$ is an edge in $G$. But for construction the vertex $u_i$ has colour $f(v_i)-d\, \mathrm{mod}\,k$ and the vertices $w_j$ have colour $f(v_j)-d\, \mathrm{mod}\,k$, so the connection maintain the same distances over the edges $v_iv_j\in E(G)$. Hence, it is a proper $(k,d)-$colouring of $\f(G)$ and if $\frac{k}{d}$ is minimal in $G$ so it's minimal in $\f(G)$.
\end{proof}

\section{$J$-Paucity Number of Graphs}

In view of our results on the absence of $J$-coloring for Mycielski graphs and our new constructions from the Mycielski graphs which admit $J$-colorings, we define a new graph parameter with respect to $J$-coloring as follows:

\begin{definition}{\rm 
Let $G$ be a graph which does not admit a $J$-coloring. Then, the \textit{$J$-paucity number} of $G$, denoted by $\varrho(G)$, is defined as the minimum number of edges to be added to $G$ so that the reduced graph becomes $J$-colorable with respect to a $(\delta(G)+1)$-coloring of $G$.
}\end{definition}

\noi In the following theorem, we determine the $J$-paucity number of paths.

\begin{theorem}
$\varrho(\mu(P_n))=n$.
\end{theorem}
\begin{proof}
Note that for $\delta(\mu(P_n))=2$ and hence we have to find the minimum number of edges to be added to $\mu(P_n)$ so that the reduced graph becomes $J$-colorable using $3$ colors. For this, first assign colors $c_1$ and $c_2$ alternatively to the vertices $v_1,v_2,\ldots,v_n$. Now color the vertices $u_i$ such that $u_i$ and its twin vertex $v_i$ have the same color. Since the vertex $w$ is adjacent to all $u_i$'s, it can be seen that it must have a different color, say $c_3$ (see Figure \ref{Fig-cmp7}).

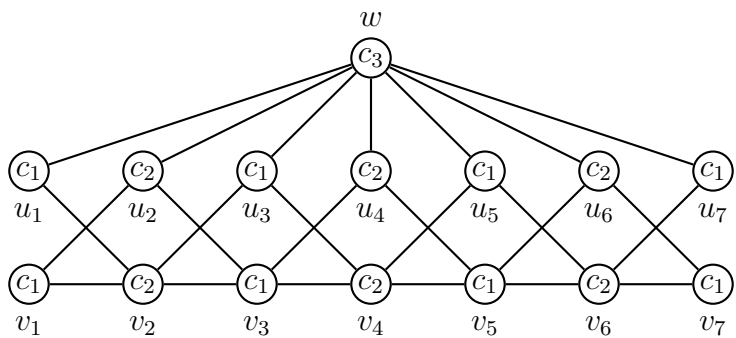
\begin{figure}[h!]
\centering
\begin{tikzpicture}[auto,node distance=1.5cm,
thick,main node/.style={circle,draw,font=\sffamily\Large\bfseries}]

\vertex [label=below:$v_1$] (1) {$c_1$};
\vertex [label=below:$v_2$] (2) [right of=1] {$c_2$};
\vertex [label=below:$v_3$] (3) [right of=2] {$c_1$};
\vertex [label=below:$v_4$] (4) [right of=3] {$c_2$};
\vertex [label=below:$v_5$] (5) [right of=4] {$c_1$};
\vertex [label=below:$v_6$] (6) [right of=5] {$c_2$};
\vertex [label=below:$v_7$] (7) [right of=6] {$c_1$};

\vertex [label=below:$u_1$] (8) [above of=1] {$c_1$};
\vertex [label=below:$u_2$] (9) [above of=2] {$c_2$};
\vertex [label=below:$u_3$] (10) [above of=3] {$c_1$};
\vertex [label=below:$u_4$] (11) [above of=4] {$c_2$};
\vertex [label=below:$u_5$] (12) [above of=5] {$c_1$};
\vertex [label=below:$u_6$] (13) [above of=6] {$c_2$};
\vertex [label=below:$u_7$] (14) [above of=7] {$c_1$};

\vertex [label=above:$w$] (15) [above of=11] {$c_3$};

\path[every node/.style={font=\sffamily\small}]
(1) edge node {} (2)
edge node {} (9)
(2) edge node {} (3)
edge node {} (8)
edge node {} (10)
(3) edge node {} (4)
edge node {} (9)
edge node {} (11)
(4) edge node {} (5)
edge node {} (10)
edge node {} (12)
(5) edge node {} (6)
edge node {} (11)
edge node {} (13)
(6) edge node {} (7)
edge node {} (12)
edge node {} (14)
(7) edge node {} (13)
(15) edge node {} (8)
edge node {} (9)
edge node {} (10)
edge node {} (11)
edge node {} (12)
edge node {} (13)
edge node {} (14);
\end{tikzpicture}
\caption{A $3$-coloring of $\mu(P_7)$.}
\label{Fig-cmp7}
\end{figure}

We notice the following points in this context:
\begin{enumerate}\itemsep0mm
\item[(i)] No vertex $v_i$ in $V(\mu(P_n))$ belongs to a rainbow neighbourhood of $\mu(P_n)$, as none of them is adjacent to a vertex having color $c_3$;
\item[(ii)] Every vertex $u_i$ with color $c_2$ is adjacent to at least one vertex $v_j$ with color $c_2$ and the vertex $w$ with color $c_3$, thus belonging to some rainbow neighbourhood in $\mu(P_n)$.
\item[(iii)] Every vertex $u_j$ with color $c_2$ is adjacent to at least one vertex $v_k$ with color $c_1$ the vertex $w$ with color $c_3$, thus belonging to some rainbow neighbourhood in $\mu(P_n)$.   
\item[(iv)] The vertex $w$, being adjacent to all vertices $u_i$, belongs to some rainbow neighbourhoods in $\mu(P_n)$.
\end{enumerate}

Therefore, from the above arguments, what we need is to draw edges from the vertices $v_i$ to the vertex $w$ so that they also are in  some rainbow neighbourhoods of $G$. Therefore, $\varrho(\mu(P_n))=n$.
\end{proof} 

\begin{theorem}
$\varrho(\mu(C_n))=n+2r$, where $r\in \N$ is given by $n\equiv r({\rm mod}\,3)$.
\end{theorem}
\begin{proof}
Since $\delta(\mu(C_n))=3$, the maximum number of colors in its $J$-coloring is $4$. Hence, we have to find the minimum number of edges to be added to $\mu(P_n)$ so that the reduced graph becomes $J$-colorable using $4$ colors. Here we have to consider the following cases:

\textit{Case-1}: Let $n\equiv 0 ({\rm mod}\,3)$. Then, we can assign colors $c_1,c_2$ and $c_3$ alternatively to the vertices $v_1,v_2,\ldots,v_n$. As mentioned in the previous result, we can color the vertices $u_i$ such that $u_i$ and its twin vertex $v_i$ have the same color. Since the vertex $w$ is adjacent to all $u_i$'s, it must have a different color, say $c_4$ (see Figure \ref{Fig-mpcmc9}). In this case, all vertices $u_i$ and the vertex $w$ will belong to some rainbow neighbourhoods of $\mu(C_n)$, but no vertex $v_i$ has an adjacent vertex having colour $c_4$. So, we need to draw edges from all $v_i;\ 1\le i\le n$ to the vertex $w$ in order to include them in some rainbow neighbourhoods of $\mu(C_n)$.

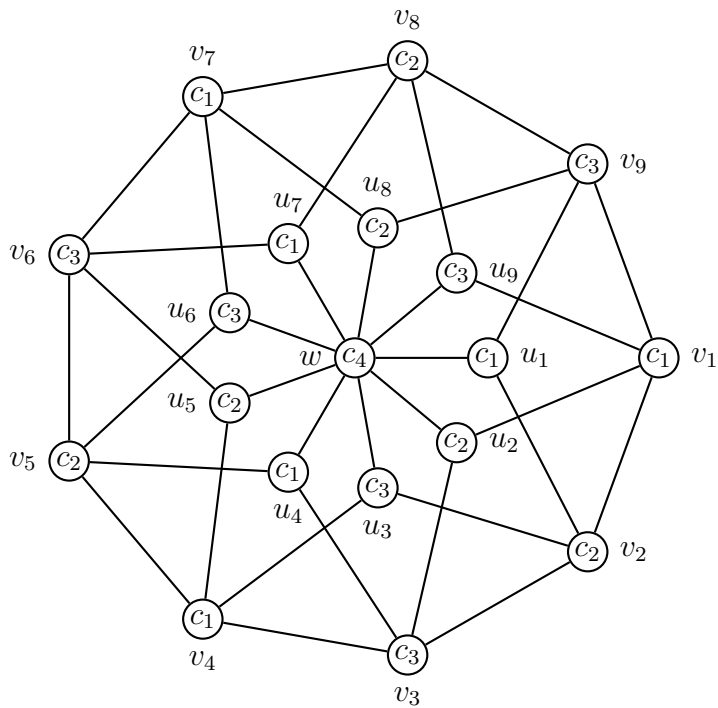
\begin{figure}[h!]
\centering
\begin{tikzpicture}[auto,node distance=2cm,
thick,main node/.style={circle,draw,font=\sffamily\Large\bfseries}]
\vertex (w) at (0:0) [label=left:$w$]{$c_4$};
\vertex (v1) at (0:4) [label=right:$v_{1}$]{$c_1$};
\vertex (v2) at (320:4) [label=right:$v_{2}$]{$c_2$};
\vertex (v3) at (280:4) [label=below:$v_{3}$]{$c_3$};
\vertex (v4) at (240:4) [label=below:$v_{4}$]{$c_1$};
\vertex (v5) at (200:4) [label=left:$v_{5}$]{$c_2$};
\vertex (v6) at (160:4) [label=left:$v_{6}$]{$c_3$};
\vertex (v7) at (120:4) [label=above:$v_{7}$]{$c_1$};
\vertex (v8) at (80:4) [label=above:$v_{8}$]{$c_2$};
\vertex (v9) at (40:4) [label=right:$v_{9}$]{$c_3$};
\vertex (u1) at (0:1.75) [label=right: $u_{1}$]{$c_1$};
\vertex (u2) at (320:1.75) [label=right:$u_{2}$]{$c_2$};
\vertex (u3) at (280:1.75) [label=below:$u_{3}$]{$c_3$};
\vertex (u4) at (240:1.75) [label=below:$u_{4}$]{$c_1$};
\vertex (u5) at (200:1.75) [label=left:$u_{5}$]{$c_2$};
\vertex (u6) at (160:1.75) [label=left:$u_{6}$]{$c_3$};
\vertex (u7) at (120:1.75) [label=above:$u_{7}$]{$c_1$};
\vertex (u8) at (80:1.75) [label=above:$u_{8}$]{$c_2$};
\vertex (u9) at (40:1.75) [label=right:$u_{9}$]{$c_3$};
\path 
(v1) edge (v2)
(v2) edge (v3)
(v3) edge (v4)
(v4) edge (v5)
(v5) edge (v6)
(v6) edge (v7)
(v7) edge (v8)
(v8) edge (v9)
(v9) edge (v1)

(u1) edge (v2)
(u1) edge (v9)
(u2) edge (v1)
(u2) edge (v3)
(u3) edge (v2)
(u3) edge (v4)
(u4) edge (v3)
(u4) edge (v5)
(u5) edge (v4)
(u5) edge (v6)
(u6) edge (v5)
(u6) edge (v7)
(u7) edge (v6)
(u7) edge (v8)
(u8) edge (v7)
(u8) edge (v9)
(u9) edge (v1)
(u9) edge (v8)

(w) edge (u1)
(w) edge (u2)
(w) edge (u3)
(w) edge (u4)
(w) edge (u5)
(w) edge (u6)
(w) edge (u7)
(w) edge (u8)
(w) edge (u9)
;
\end{tikzpicture}
\caption{A minimal proper coloring of $\mu(C_9)$}
\label{Fig-mpcmc9}
\end{figure}

\textit{Case-2}: Let $n\equiv 1 ({\rm mod}\,3)$. Then, we can assign colors $c_1,c_2$ and $c_3$ alternatively to the vertices $v_1,v_2,\ldots,v_{n-1}$. The vertex $v_n$ can be colored only by $c_2$, as it is adjacent to $v_1$ with color $c_1$ and to $v_{n-1}$ with color $c_3$. Here, we notice that the vertex $v_1$ is not adjacent to any vertex having color $c_3$. Here, we need to draw an edge between $v_1$ and one of the vertices having color $c_3$. 

If we label the vertices $u_i$ in such a way that the twin vertices have the same color, then as in the case of $v_1$, the vertex $u_1$ is not adjacent to any vertex of color $c_3$. Hence, we need to draw an edge from $u_1$ to any one of the vertices having color $c_3$.

Since $w$ is adjacent to all $u_i$, $w$ must have the fourth color $c_4$.  Since every vertex $u_i$ is adjacent to $w$, all these vertices,(except $u_1$) belong to some rainbow neighbourhood of $\mu(C_n)$. (Also, note that when we draw an edge from $u_1$ to a vertex having color $c_3$, it will also belong to some rainbow neighbourhood).

Since no vertex $v_i$ is adjacent to a vertex having colour $c_4$, each of them is to be connected to the vertex $w$ by a new edge. Therefore, in this case $\varrho(\mu(C_n))=n+2$.

\textit{Case-3}: Let $n\equiv 2 ({\rm mod}\,3)$. Then, we can assign colors $c_1,c_2$ and $c_3$ alternatively to the vertices $v_1,v_2,\ldots,v_{n-2}$. Then, the vertex $v_{n-1}$ gets the color  $c_1$, the vertex $v_1$ can have the color color $c_2$ Note that the vertex $v_1$ and $v_n$ are not adjacent to any vertex having color $c_3$. Here, we need to draw one edge each from $v_1$ and $v_2$ to some vertices having color $c_3$. 

If we label the vertices $u_i$ in such a way that the twin vertices have the same color, then as in the case of $v_1$ and $v_2$, the vertices $u_1$ and $u_n$ will not be adjacent to any vertex of color $c_3$. Hence, we need to draw one edge each from $u_1$ and $u_3$ to some of the vertices having color $c_3$.

The vertex $w$ gets the color $c_4$ and as mentioned in the above cases, we need to draw edges from all vertices $v_i$ to $w$ so that all vertices in $\mu(C_n)$ belong to some rainbow neighbourhoods in $\mu(C_n)$. Therefore, in this case $\varrho(\mu(C_n))=n+4$.
\end{proof}


\section{Conclusion}

In this paper, we have proved that the Mycielskian of any graph $G$ will not have a $J$-coloring, irrespective of whether $G$ has a $J$-coloring or not. We have also checked the existence of $J$-coloring for certain new Mycielski type graphs constructed from certain graphs. There is a wide scope for further studies in this area by exploring for new and related graph constructions.

We have also investigated the possibility of defining $J$-colorings for given graphs by adding new edges between their non-adjacent vertices. Furthermore, we have determined the minimum number of such edges to be introduced for the Mycielskian of paths and cycles. The studies in this area for more graph classes and more derived graphs are also promising.


\section*{Declaration}
The authors declare that they have no competing interests. Both the authors contributed significantly in writing this article. The authors read and approved the final manuscript.





\begin{thebibliography}{96}

\bibitem{BM1} J. A. Bondy and U. S. R. Murty, {\bf Graph theory}, Springer, Berlin, 2008.

\bibitem{CZ1} G. Chartrand and P. Zhang, \textbf{Chromatic graph theory}, CRC Press, 2009.

\bibitem{FH1} F. Harary, \textbf{Graph theory}, Narosa Publishing House, New Delhi, 2001.

\bibitem{JT1} T. R. Jensen and B. Toft, {\bf Graph coloring problems}, John Wiley \& Sons, 1995.

\bibitem{KSM} J. Kok, N. K. Sudev, U. Mary, {\it On chromatic Zagreb indices of certain graphs}, Discrete Math. Algorithm. Appl., \textbf{9}(1)(2017), 1750014:1-14, DOI: 10.1142/S1793830917500148.

\bibitem{KS1} J. Kok, N. K. Sudev {\it $J$-coloring of graph operations}, preprint.

\bibitem{KSJ} J. Kok, N. K. Sudev, M. K. Jamil, {\it Rainbow neighbourhood number of graphs}, preprint.

\bibitem{MK1}  M. Kubale, {\bf Graph colorings}, American Mathematical Society, 2004.

\bibitem{LWLG} W. Lin, J. Wu, P. C. B. Lam, G. Gu, \textit{Several parameters of generalized Mycielskians}, Discrete Appl. Math., \textbf{154}(8)(2006), 1173-1182, DOI:10.1016/j.dam.2005.11.001.

\bibitem{NKS} N.K. Sudev, \textit{On certain $J$-colouring parameters of graphs}, preprint, arXiv:1612.04194.

\bibitem{SSS} C. Susanth, S.J. Kalayathankal, N.K. Sudev, \textit{Rainbow neighbourhood number of certain Mycielski type graphs}, Global Stoch. Anal., to appear.

\bibitem{AV} A. Vince, \textit{Star chromatic number}, J. Graph Theory 12 (1988), 551-559.

\bibitem{DBW} D. B. West, {\bf Introduction to graph theory}, Pearson Education Inc., 2001.

\bibitem{XZ} X. Zhu, \textit{The circular chromatic number: A survey}, Disc. Math., 229 (2001), 371-410.
\end{thebibliography}
\end{document}